\documentclass[a4paper]{amsart}
\usepackage{amsmath,amsthm,amssymb}
\usepackage{graphicx} 
\usepackage{tcolorbox}
\usepackage{tikz-cd}
\usetikzlibrary{arrows.meta,positioning,decorations.markings}
\usepackage[margin=1in]{geometry}
\usepackage{enumerate}
\usepackage{hyperref}

\DeclareMathOperator{\bbz}{\mathbb{Z}}

\DeclareMathOperator{\bbf}{\mathbb{F}}

\newcommand{\FFalgebra}[1]{\bbf_p[\![#1]\!]}

\title{Revisiting fully residually free Demushkin groups}
\author{Henrique Souza and Pavel Zalesskii}
\date{\today}

\newtheorem{thm}{Theorem}[section]
\newtheorem{prop}[thm]{Proposition}
\newtheorem{lem}[thm]{Lemma}
\newtheorem{cor}[thm]{Corollary}
\theoremstyle{definition}

\newtheorem{defn}[thm]{Definition}
\newtheorem{question}[thm]{Question}
\newtheorem{rmk}[thm]{Remark}

\newtcolorbox{todobox}{colback=red!5!white,colframe=red!75!black}

\newtcolorbox{henbox}{colback=red!5!white,colframe=red!75!black,title=Henrique}
\newtcolorbox{abox}{colback=blue!5!white,colframe=blue!75!black,title=Andrei}
\newtcolorbox{pbox}{colback=green!5!white,colframe=green!75!black,title=Pavel}

\begin{document}

\begin{abstract}
    We establish new examples and non-examples of pro-\(p\) limit groups among the class of Demushkin groups, that is, pro-\(p\) Poincaré duality groups of dimension \(2\).
\end{abstract}

\maketitle

\section{Introduction}

The class of limit groups is a well studied class of finitely generated discrete groups that admits various descriptions:
\begin{itemize}
    \item \(G\) is a limit group if it is \emph{fully residually free}, that is, given any finite subset \(X \subseteq G\), there exists a homomorphism \(\varphi\colon G \to F\) where \(F\) is a non-abelian free group whose restriction to \(X\) is injective.
    \item \(G\) is a limit group if it has the same \emph{existential theory} of a free group, that is, the set of first order sentences in the language of groups that only use the existential quantifier \(\exists\) that hold in \(G\) are precisely the ones that hold in a non-abelian free group.
    \item \(G\) is a limit group if it is a subgroup of an iterated sequence of \emph{centralizer extensions} starting from free groups, that is, free products with amalgamation \(H *_C (C \times \bbz^n)\) where \(C\) is a cyclic self-centralized subgroup of \(H\).
\end{itemize}
This class of groups attracted a lot of attention for their role in the solution of the Tarski problem (\cite{kharlampovichElementaryTheoryFree2006,selaDiophantineGeometryGroups2006}), as well as their interesting group-theoretic properties (\cite[Prop. 1.1]{alibegovicLimitGroupsAre2006}). A non-trivial example of a non-free limit group is the fundamental group of an orientable closed surface, as well as the fundamental group of a non-orientable closed surface of Euler characteristic smaller than \(-1\).

The study of the pro-\(p\) analogues of limit groups goes back to the works of D. Kochloukova and the second author. In \cite{kochloukovaPropAnaloguesLimit2011}, they introduced the class \(\mathcal{L}\) of finitely generated pro-\(p\)  groups defined recursively as follows: \(\mathcal{L}\) contains the class \(\mathcal{G}_0\) of finitely generated free pro-\(p\) groups, and for all \(n > 0\) it contains all finitely generated subgroups of the class \(\mathcal{G}_n\) of centralizer extensions \(G_{n-1} \amalg_C A\) where \(G_{n-1} \in \mathcal{G}_{n-1}\), \(C\) is a self-centralizing procyclic subgroup of \(G_{n-1}\) and \(A \simeq \bbz_p^k\) is a free abelian pro-\(p\) group of finite rank containing \(C\) as a direct summand. The smallest \(n\) for which a \(G\) in \(\mathcal{L}\) embeds in a group of \(\mathcal{G}_n\) is called the \emph{weight} of \(G\). For now, we shall call the pro-\(p\) groups in the class \(\mathcal{L}\) \emph{pro-\(p\) limit groups}. 

The pro-\(p\) limit groups share many interesting group-theoric properties with their discrete counterparts: they are free-by-(torsion-free nilpotent), have finite cohomological dimension, are of type \(FP_\infty\) and have non-positive Euler characteristic, they are either abelian or such that any finitely generated closed normal subgroup is open and any \(2\)-generated subgroup is free pro-\(p\) or isomorphic to \(\bbz_p^2\). In \cite{kochloukovaPropAnaloguesLimit2011}, they listed a few properties of abstract limit groups that are not yet clear for their pro-\(p\) analogues. Amongst them, we point out a particular one:

\begin{question}\label{quest-free-pro-p} Is every pro-\(p\) limit group residually free pro-\(p\)?
\end{question}

We also point out that it is not known in the pro-\(p\) setting whether every fully residually free pro-\(p\) group is a pro-\(p\) limit group and vice-versa. We are not aware of any study on the existential theory of free pro-\(p\) groups. Also, in contrast with the discrete case, much less is known in terms of examples. For instance:

\begin{question} Let \(\Gamma\) be an orientable surface group of genus \(g\) and \(G = \widehat{\Gamma}_p\) be its pro-\(p\) completion. Is \(G\) a pro-\(p\) limit group?
\end{question}

In \cite{kochloukovaPropAnaloguesLimit2011}, the answer as shown to be yes provided the genus \(g\) is even, in which case they also prove that these groups are fully residually free pro-\(p\). The purpose of this short note is to provide a way to generalize this class of examples. We recall that a Demushkin group is a pro-\(p\) Poincaré duality group of dimension \(2\), and we refer the reader to Section~\ref{sec-demushkin} for the precise definition of the torsion-invariant and the orientation character.

\begin{thm}\label{thm-demushkin-is-limit} Let \(G\) be a pro-\(p\) Demushkin group with  minimal number of generators \(n\), torsion-invariant \(q \in \{0\,,p\,,p^2\,,\cdots\}\) and orientation character \(\chi\colon G \to \mathbb{Z}_p^\times\).
\begin{enumerate}
    \item If \(n\) is divisible by \(4\), then \(G\) is a pro-\(p\) limit group.
    \item If \(p=2\) and \(n > 2\) is even, then \(G\) is a pro-\(2\) limit group.
\end{enumerate}
Moreover, in both cases above, \(G\) is (fully) residually free pro-\(p\).
\end{thm}

In general, one cannot drop the assumption that \(n\) is even, as the pro-\(2\) Demushkin group \[D = \langle x_1,x_2,x_3\mid x_1^2[x_2,x_3]\rangle \simeq \langle a,b,c\mid a^2b^2c^2 = 1\rangle\] is neither a pro-\(2\) limit group nor residually free pro-\(2\). We will prove a stronger statement:

\begin{thm}\label{thm-three-gen-is-not} Any homomorphism from \(D\) to a pro-\(2\) limit group has abelian image.
\end{thm}

Since \([x_2,x_3]\) is a non-trivial element of \(D\) that cannot be mapped non-trivially to a free pro-\(2\) group, \(D\) is not residually free pro-\(2\) nor it can be embedded in a pro-\(2\) limit group. In what follows, we first collect in Section~\ref{sec:prelim} some results on the pro-\(p\) analogue of Bass-Serre theory of groups acting that will be need in order to prove Theorems~\ref{thm-demushkin-is-limit} and~\ref{thm-three-gen-is-not}. Then, in Section~\ref{sec-demushkin}, we recall the definition of Demushkin groups and its associated invariants, as well as their classification in terms of said invariants. After some discusssion on procyclic splittings of Demushkin groups, both theorems are proven.

\section{Preliminaries on pro-\texorpdfstring{\(p\)}{p} groups acting on pro-\texorpdfstring{\(p\)}{p} trees}\label{sec:prelim}

In this section, we recall some notations and facts about the pro-\(p\) analogues of Bass-Serre theory of groups acting on trees that will be used throughout the paper. Our standard refences on the subject are Ribes-Zalesskii's Chapter in \cite{horizons} and  L. Ribes' book \cite{ribesProfiniteGraphsGroups2017}. 

A profinite graph is a compact, Hausdorff and totally disconnected space \(\Gamma\) endowed with a closed subspace \(V\Gamma\) of vertices and two continuous functions \(d_0,d_1\colon \Gamma \to V(\Gamma)\) extending the identity on \(V(\Gamma)\). The elements of \(E(\Gamma) =\Gamma \smallsetminus V(\Gamma)\) are called the edges of \(\Gamma\), \(d_0\) and \(d_1\) are called the incidence maps and for each edge \(e \in E(\Gamma)\) we say that \(d_0(e)\) is the beginning of \(e\) and \(d_1(e)\) is the end-point of \(e\). 
We will only consider profinite graphs where \(E(\Gamma)\) is also a closed subset of \(\Gamma\) -- observe that this property is preserved under passing to closed subgraphs (but not preserved by quotients).

A morphism \(\alpha\colon \Gamma \to \Delta\) of profinite graphs is a continuous map that commutes with \(d_0\) and \(d_1\). Any profinite graph is isomorphic to an inverse limit of finite discrete graphs. We say that a profinite graph is connected if all of its finite quotient graphs are connected. 

Associated to any profinite graph \(\Gamma\) we have a sequence of morphisms of free profinite \(\bbf_p\)-vector spaces:
\begin{equation}\label{short-sequence-graph}
    0 \to \FFalgebra{E(\Gamma)} \overset{d}{\to} \FFalgebra{V(\Gamma)} \overset{\epsilon}{\to} \bbf_p \to 0
\end{equation}
where \(\epsilon\) maps every element of \(V(\Gamma)\) to \(1 \in \bbf_p\) and \(d(e) = d_1(e) -  d_0(e)\) for each \(e \in E(\Gamma)\). A profinite graph is connected if and only if \(\ker(\epsilon) = \operatorname{im}(d)\) (\cite[Prop. 2.3.2]{ribesProfiniteGraphsGroups2017}), and we say that a connected profinite graph is a pro-\(p\) tree if the sequence (\ref{short-sequence-graph}) is exact. A finite graph is a pro-\(p\) tree if and only if it is a tree as an abstract graph. If \(\alpha\colon \Gamma \to \Delta\) is a surjective morphism of connected graphs with \(\Delta\) a finite graph and \(T\) is a maximal subtree of \(\Delta\), a fundamental \(0\)-section of \(\alpha\) (with respect to \(T\)) is a function \(\sigma\colon \Delta \to \Gamma\) such that \(\sigma_{\mid T}\) is an isomorphism of \(T\) onto its image and \(\partial_0(\sigma(e)) \in \sigma(\Delta)\) for all \(e \in E(\Gamma)\). A fundamental \(0\)-section always exists when \(\Delta = G\backslash \Gamma\) for some profinite group \(G\) acting on \(\Gamma\) with a finite quotient (\cite[Prop. 3.4.5]{ribesProfiniteGraphsGroups2017})

If \(\Delta\) is a finite graph, a (proper) graph \(\mathcal{G}\) of pro-\(p\) groups on \(\Delta\) is a map that attaches to each vertex \(v \in V(\Delta)\) a pro-\(p\) group \(\mathcal{G}(v)\) and to each edge \(e \in E(\Delta)\) a  pro-\(p\) group \(\mathcal{G}(e)\) together with two injective continuous homomorphisms \(\partial_0\colon \mathcal{G}(e) \to \mathcal{G}(d_0(e))\) and \(\partial_1\colon \mathcal{G}(e) \to \mathcal{G}(d_1(e))\). The fundamental group \(\pi_1(\mathcal{G},\Delta)\) of the graph of groups \(\mathcal{G}\) is defined as follows: we fix a maximal subtree \(T\) of \(\Delta\), and we take \(F\) to be the free pro-\(p\) product of all the vertex groups \(\mathcal{G}(v)\), \(v \in V(\Delta)\), with the free pro-\(p\) group on all the letters \(t_e\), \(e \in E(\Delta)\). Then, \(\pi_1(\mathcal{G},\Delta)\) is defined as the quotient \(F/N\) where \(N\) is the closed normal subgroup of \(F\) generated by all the elements \(t_e\) for \(e \in E(T)\) and \(t_e\partial_0(x)t_{\overline{e}}\partial_1(x)^{-1}\) for \(e \in E(\Delta)\) and \(x \in \mathcal{G}(e)\). The isomorphism class of \(\pi_1(\mathcal{G},\Delta)\) does not depend on the choice of \(T\). Associated to this graph of groups \(\mathcal{G}\) one has a pro-\(p\) tree \(\Gamma\) called the standard Bass-Serre tree of \(\mathcal{G}\) (see \cite[Sec. 2.3]{ribesProfiniteGraphsGroups2017}). This tree can be endowed with an action of \(\pi_1(\mathcal{G}, \Delta)\) for which \(\pi_1(\mathcal{G}, \Delta)\backslash \Gamma \simeq \Delta\) as graphs. Particular cases of fundamental groups of graphs of pro-\(p\) groups are free pro-$p$ products with amalgamation \(A \amalg_C B\) and HNN-extensions \(\mathrm{HNN}(A,C,t)\).

If \(\Gamma\) is a pro-\(p\) tree and \(G\) is a pro-\(p\) group acting on \(\Gamma\) without inversion of edges with finite quotient graph \(\Delta = G\backslash \Gamma\), we define the induced graph of groups \(\mathcal{G}\) on \(\Delta\) as follows: we let \(T\) be a maximal subtree of \(\Delta\), \(\sigma\colon \Delta \to \Gamma\) be a fundamental \(0\)-section with respect to \(T\) and for each \(x \in \Gamma\) we let \(\mathcal{G}(x) = \mathrm{Stab}_G(\sigma(x))\). For all \(e \in E(\Gamma)\), the map \(\partial_0\colon \mathcal{G}(e) \to \mathcal{G}(d_0(e))\) is simply the inclusion map \(\mathrm{Stab}_G(\sigma(e)) \to \mathrm{Stab}_G(d_0(\sigma(e)))\). For all \(e \in E(\Gamma)\), we have that \(\sigma(d_1(e))\) and \(d_1(\sigma(e))\) are in the same \(G\)-orbit, so we choose an element \(\gamma_e \in G\) such that \(\gamma_e\sigma(d_1(e)) = d_1(\sigma(e))\), taking care to choose \(\gamma_e = 1\) if \(e \in T\). Then, we let \(\partial_1\colon \mathcal{G}(e) \to \mathcal{G}(d_1(e))\) be the composition of the inclusion map \(\mathrm{Stab}_G(\sigma(e)) \to \mathrm{Stab}_G(d_1(\sigma(e)))\) with the conjugation by \(\gamma_e^{-1}\) map \[\mathrm{Stab}_G(d_1(\sigma(e))) = \gamma_e\mathrm{Stab}_G(\sigma(d_1(e)))\gamma_e^{-1} \to \mathrm{Stab}_G(\sigma(d_1(e)))\,.\] 

The pro-\(p\) analogue of the fundamental structure theorem of Bass-Serre theory says (\cite[Thm. 6.6.1]{ribesProfiniteGraphsGroups2017}): if \(G\) is a pro-\(p\) group acting on a pro-\(p\) tree \(\Gamma\)  such that \(\Delta = G\backslash \Gamma\) is finite, then for the graph of groups \(\mathcal{G}\) on \(\Delta\) constructed above the map \(\varphi\colon \pi_1(\mathcal{G},\Delta) \to G\) defined by the inclusion maps on each vertex group \(\mathcal{G}(v)\) and by mapping \(t_e \mapsto \gamma_e\) is an isomorphism of profinite groups (for any suitable choice of \(\gamma_e\)).


We recall that an action of a pro-\(p\) group \(G\) on a pro-\(p\) tree \(\Gamma\) is \emph{\(k\)-acylindrical} if for every non-trivial \(g \in G\) the fixed subtree \(\Gamma^g\) has diameter at most \(k\). We will need the following result of Z. Chatzidakis and the second author:

\begin{thm}[{\cite[Thm. 5.4]{chatzidakisPropGroupsActing2022}}]\label{thm-chatzidakis} Let \(G\) be a finitely generated pro-\(p\) group acting \(k\)-acylindrically on a pro-\(p\) tree \(\Gamma\) with finitely many maxmal vertex stabilizers up to conjugation. Then:
\begin{enumerate}[(i)]
    \item The closure \(\overline{D}\) of \(D = \{t \in \Gamma\colon \mathrm{Stab}_G(t) \neq 1\}\) is a profinite \(G\)-invariant subgraph of \(\Gamma\) having finitely many connected components \(\Sigma_1\,,\ldots\,,\Sigma_m\) up to translation.
    \item For the setwise stabilizer \(G_i = \mathrm{Stab}_G(\Sigma_i)\), the quotient graph \(G_i\backslash \Sigma_i\) has finite diameter and \(\Sigma_i\) contains a \(G_i\)-invariant subtree \(\Gamma_i\) such that \(G_i\backslash \Gamma_i\) is finite.
    \item \(G = \coprod_{i=1}^m G_i \amalg F\) is a free pro-\(p\) product, where \(F\) is a free pro-\(p\) group acting freely on \(T\).
\end{enumerate}
\end{thm}

We will also need the following result of I. Snopche and the second author:


\begin{prop}[{\cite[Prop. 3.3]{snopceSubgroupPropertiesPro$p$2014}}]\label{prop-finitely-many-stabilizers} Let \(G\) be a finitely generated pro-\(p\) group acting on a pro-\(p\) tree \(\Gamma\) with procyclic edge stabilizers. Then, the maximal vertex stabilizers in \(G\) are finitely generated and there are only finitely many of them in \(G\) up to conjugation.
\end{prop}

\section{Demushkin groups}\label{sec-demushkin}

We recall that a pro-\(p\) group \(G\) is called a Demushkin group if its \(\bbf_p\) cohomology satisfies Poincaré duality in dimension \(2\), that is:
\begin{itemize}
    \item \(\dim_{\bbf_p} H^1(G,\bbf_p) < \infty\);
    \item \(\dim_{\bbf_p} H^2(G,\bbf_p) = 1\);
    \item The cup product \[H^1(G,\bbf_p) \times H^1(G,\bbf_p) \overset{\cup}{\to} H^2(G,\bbf_p)\] is a non-degenerate bilinear form.
\end{itemize}
Examples of Demushkin groups are the pro-\(p\) completions of orientable surface groups, the pro-\(2\) completions of non-orientable surface groups and the Galois groups of the maximal \(p\)-extensions of \(p\)-adic fields that contain all \(p\)-th roots of unity (\cite[Sec. 
 I.4.5, Thm. II.5.4]{serreGaloisCohomology1997}). By the interpretation of \(H^i(G,\bbf_p)\) in terms of generators and relations (\cite[Sec. I.4]{serreGaloisCohomology1997}), the first two conditions mean that \(G\) must be a finitely generated one relator pro-\(p\) group. The Demushkin groups of small rank (ie. \(p\)-adic analytic or solvable) are precisely the two-generator ones: \(\bbz/2\bbz\) and all semidirect products \(\bbz_p \rtimes \bbz_p\). Any Demushkin group requiring at least \(3\)-generators is full sized, that is, contains a non-abelian free pro-\(p\) subgroup.

Given a finitely generated one relator pro-\(p\) group \(G\) with a minimal presentation \[G = \langle x_1,\ldots,x_n\mid r\rangle\,,\] there is a unique \(q \in \{0,p,p^2,p^3,\ldots\}\) such that the abelianization \(G/[G,G]\) is isomorphic to \(\bbz_p^{n-1} \times \bbz_p/q\bbz_p\), and we call it the torsion-invariant of \(G\). If \(F = F(x_1,\ldots,x_n)\) is the free pro-\(p\) group on the basis \(x_1,\ldots,x_n\) and \(F_1 = F\), \(F_i = F_{i-1}^q[F_{i-1},F]\) denotes the lower \(q\)-central series of \(F\), we have \(r \in F_2\) and there are unique coefficients \(a_i\), \(b_{ij}\) in \(\bbz_p/q\bbz_p\) such that \[r \equiv x_1^{qa_1}\cdots x_n^{qa_n}\prod_{1\leq i < j \leq n}[x_i,x_j]^{b_{ij}} \pmod{F_3}\,.\] These coefficients are all detected by the cup product:

\begin{prop}[{\cite[Prop. 3.9.13]{neukirchCohomologyNumberFields2008}}]\label{prop-cup-product} The matrix representing the cup product \[H^1(G,\bbz_p/q\bbz_p) \times H^1(G,\bbz_p/q\bbz_p) \to H^2(G,\bbz_p/q\bbz_p)\] in the ordered basis \(\chi_1,\ldots,\chi_n\) dual to \(x_1,\ldots,x_n\) is given by \((c_{ij})\) where \[c_{ij} = \begin{cases}
    b_{ij}\text{, if }i<j\\
    -b_{ji}\text{, if }i > j\\
    -\binom{q}{2}a_i\text{, if } i = j
\end{cases}\] In particular, \(G\) is a Demushkin group if and only if the matrix \((c_{ij})\) is invertible in \(\bbz_p/q\bbz_p\).
\end{prop}

Almost two decades prior to the classification of the abstract Poincaré duality groups in dimension \(2\) by B. Eckmann and P. Linnell (\cite{eckmannPoincareDualityGroups1983}), the classification of Demushkin groups was completed by S. Demushkin, J.-P. Serre and J. Labute. We also recall that to any pro-\(p\) Poincaré duality group one can associate a canonical continuous homomorphism \(\chi\colon G \to \bbz_p^\times\), called its orientation character (\cite[Sec. I.4.5]{serreGaloisCohomology1997}). To state the classification theorem, we introduce the following definition:

\begin{defn} Let \(F = F(x_1,\ldots,x_n)\) be the free pro-\(p\) group, \(q \in \{0,p,p^2,p^3,\ldots\}\) and \(\mathbb{U}\) a closed pro-\(p\) subgroup of \(\bbz_p^\times\). We define the standard Demushkin relation \(r_1(n,q,\mathbb{U})\) as follows:
\begin{itemize}
    \item If \(q \neq 2\), then \(n\) is even, \(\mathbb{U} = 1 + q\bbz_p\) and \[r_1(n,q,\mathbb{U}) = x_1^q[x_1,x_2]\cdots [x_{n-1},x_n]\,.\]
    \item If \(q  = 2\), \(n\) is even and \(\mathbb{U} = \langle -1 + 2^f\rangle\) for some \(f \geq 2\), then \[r_1(n,q,\mathbb{U}) = x_1^{2+2^f}[x_1,x_2]\cdots [x_{n-1},x_n]\,.\]
    \item If \(q = 2\), \(n\) is even and \(\mathbb{U} = \langle -1, 1+2^f\rangle\) for some \(f \geq 2\) or \(\infty\), then \[r_1(n,q,\mathbb{U}) = x_1^2[x_1,x_2]x_3^{2^f}[x_3,x_4]\cdots [x_{n-1},x_n]\,.\]
    \item If \(q = 2\), \(n\) is odd and \(\mathbb{U} = \langle -1, 1+2^f\rangle\) for some \(f \geq 2\) or \(\infty\), then \[r_1(n,q,\mathbb{U}) = x_1^2x_2^{2^f}[x_2,x_3]\cdots [x_{n-1},x_n]\,.\]
\end{itemize}
Otherwise, we say that \(r_1(n,q,\mathbb{U})\) is undefined.
\end{defn}

\begin{thm}[{\cite{labuteClassificationDemushkinGroups1967}}]\label{thm-classification} If \(G\) is a Demushkin group with invariants \(n\), \(q\) and \(\mathbb{U} = \mathrm{im}(\chi)\), then \(r_1,(n,q,\mathbb{U})\) is defined and there exists a presentation \[G \simeq \langle x_1,\ldots,x_n \mid r_1(n,q,\mathbb{U})\rangle\,.\] Conversely, given \(n\), \(q\) and \(\mathbb{U} \leq \bbz_p^\times\) such that \(r_1(n,q,\mathbb{U})\) is defined, such a presentation defines a Demushkin group with the given invariants, and no such two are pairwise isomorphic.
\end{thm}

Observe that the torsion-invariant \(q\) is the largest power of \(p\) such that \(\mathbb{U}\) is contained in \(1+q\bbz_p\). In particular, the two invariants \(n\) and \(\mathbb{U}\) completely determine \(G\) up to isomorphism, as does \(n\) and \(q\) provided \(q \neq 2\). The pro-\(p\) completions of orientable surface groups are the ones with \(\mathbb{U} = \{1\}\), and the pro-\(2\) completions of the non-orientable surface groups are the ones with \(\mathbb{U} =  \{\pm 1\}\).

The following theorem is a particular case of a result of G. Wilkes (\cite{wilkesRelativeCohomologyTheory2019}), for which we present a different proof.

\begin{thm}\label{thm-double} Let \(n\), \(q\) and \(\mathbb{U}\) be such that the standard Demushkin relation \(r = r_1(n,q,\mathbb{U})\) is defined in \(F = F(x_1,\ldots,x_n)\). Then, the Baumslag double \(G = F \amalg_{r = \widetilde{r}} \widetilde{F}\) where \(\widetilde{F} = F(\widetilde{x}_1,\ldots,\widetilde{x}_n)\) is an isomorphic copy of \(F\) is a Demushkin group with invariants \(2n\), \(q\) and \(\mathbb{U}\).
\end{thm}
\begin{proof} Since \(r\) and \(\widetilde{r}\) are in the Frattini subgroups of \(F\) and \(\widetilde{F}\) respectively, it is immediate to check that the minimal number of generators of \(G\) is \(2n\). From the presentation of \(G\) as an amalgamated free pro-\(p\) product, it is also easy to check that its torsion-invariant is precisely \(q\). In the free pro-\(p\) group \(F \amalg \widetilde{F}\), we have \begin{equation}\label{eq-rel-double}r\widetilde{r}^{-1} \equiv \begin{cases}x_1^\alpha\widetilde{x}_1^{-\alpha}x_3^\beta\widetilde{x_3}^{-\beta} [x_1,x_2][\widetilde{x}_1,\widetilde{x}_2]^{-1} \cdots [x_{n-1},x_n][\widetilde{x}_{n-1},\widetilde{x}_n]^{-1}\text{, if }n\text{ is even}\\ x_1^2\widetilde{x}_1^{-2}x_2^{2^f}\widetilde{x}_2^{-2^f}[x_2,x_3][\widetilde{x}_2,\widetilde{x}_3]^{-1}\cdots [x_{n-1},x_n][\widetilde{x}_{n-1},\widetilde{x}_n]^{-1}\text{ if }n\text{ is odd}\end{cases}\end{equation} modulo the third term of the lower \(q\)-central series of \(F \amalg \widetilde{F}\). In the basis of \(H^1(G)\) dual to the basis \(x_i,\widetilde{x}_j\), it follows from Proposition~\ref{prop-cup-product} and the congruence~(\ref{eq-rel-double}) that the matrix representing the cup-product is block diagonal with blocks \[ \pm\begin{pmatrix}0 & 1 \\ -1 & 0\end{pmatrix}\text{ and }\begin{pmatrix}1 & 0 \\ 0 & 1\end{pmatrix} \text{ if }n\text{ is odd}\,, \pm\begin{pmatrix}
    \binom{\alpha}{2} & 1\\ -1 & 0
\end{pmatrix} \text{ if }n\text{ is even}\,,\] which are all invertible modulo \(q\). Hence, \(G\) is a Demushkin group with invariants \(2n\) and \(q\). It only remains to determine its invariant \(\operatorname{im}(\chi)\).

If \(q \neq 2\), then \(\operatorname{im}(\chi) = 1 + q\mathbb{Z}_p = \mathbb{U}\) and we are done, so assume now that \(q = 2\) and let \(\chi\colon G \to 1+2\bbz_2\) be the orientation character of \(G\). It was shown in \cite[Thm. 4]{labuteClassificationDemushkinGroups1967} that \(\chi\) is the \emph{unique} character of \(G\) such that, if \(I = \bbz_p\) is the \(G\)-module with the continuous action induced by \(\chi\), then any function \(\{x_1,\widetilde{x}_1,\ldots,x_{n},\widetilde{x}_{n}\} \to I\) extends continuously to a unique derivation \(G \to I\). If \(D\colon G \to I\) is a derivation and \([a,b] = a^{-1}b^{-1}ab \in G\), we have \[D([a,b]) = \chi(a^{-1}b^{-1})((1-\chi(b))D(a) +(\chi(a)-1)D(b))\,.\] Moreover, we have \[D(a^k) = (1+\chi(a)+\chi(a)^2 + \cdots + \chi(a)^{k-1})D(a)\] for any \(k \geq 0\). Finally, observe that since \([G,G] \in \ker \chi\), we have \[D([a,b][c,d]) = D([a,b]) + D([c,d])\] for any \(a,b,c,d \in G\).

Now, let \(D_i\colon G \to I\) (resp. \(\widetilde{D}_i\colon G \to I\)) be the derivation that is equal to \(1\) on \(x_i\) (resp. \(\widetilde{x}_i\)) and vanishes on all other generators. We have:
\begin{itemize}
    \item If \(n\) is even and \(\mathbb{U} = \langle -1 + 2^f\rangle\), then 
    \begin{align*}
        D_2(r\widetilde{r}^{-1}) = \chi(x_1)^{2+2^f}\chi(x_1^{-1}x_2^{-1})(\chi(x_1)-1) = 0 & \implies \chi(x_1) = 1\,,\\
        D_1(r\widetilde{r}^{-1}) = 2+2^f + \chi(x_2^{-1})(1-\chi(x_2)) = 0 &\implies \chi(x_2) = (-1 -2^f)^{-1}\,,\\
        D_i(r\widetilde{r}^{-1}) = 0 & \implies \chi(x_i) = 1\text{ for }i>2\,,\\
        \widetilde{D}_2(r\widetilde{r}^{-1}) = \chi(\widetilde{x}_2^{-1}\widetilde{x}_1^{-1})(1-\chi(\widetilde{x}_1)) = 0 &\implies \chi(\widetilde{x}_1) = 1\,,\\
        \widetilde{D}_1(r\widetilde{r}^{-1}) = \chi(\widetilde{x}_2^{-1})(\chi(\widetilde{x}_2)-1) - 2-2^f = 0 & \implies \chi(\widetilde{x}_2)= (-1-2^f)^{-1}\,,\\
        \widetilde{D}_i(r\widetilde{r}^{-1}) = 0 &\implies \chi(\widetilde{x}_i) = 1\text{ for }i > 2\,.\end{align*}
    \item If \(n\) is even and \(\mathbb{U} = \langle -1, 1+2^f\rangle\), then
    \begin{align*}
        D_2(r\widetilde{r}^{-1}) = \chi(x_1)^2\chi(x_1^{-1}x_2^{-1})(\chi(x_1)-1) = 0 & \implies \chi(x_1) = 1\,,\\
        D_1(r\widetilde{r}^{-1}) = 2 + \chi(x_2^{-1})(1-\chi(x_2))&\implies \chi(x_2) = -1\,,\\
        D_4(r\widetilde{r}^{-1}) = \chi(x_3)^{2^f}\chi(x_3^{-1}x_4^{-1})(\chi(x_3)-1) = 0 &\implies \chi(x_3) =1\,,\\
        D_3(r\widetilde{r}^{-1}) = 2^f + \chi(x_4^{-1})(1-\chi(x_4)) = 0 &\implies \chi(x_4) = -(1+2^f)^{-1}\,,\\
        D_i(r\widetilde{r}^{-1}) = 0 &\implies \chi(x_i) = 1\text{ for }i>4\,,\\
        \widetilde{D}_4(r\widetilde{r}^{-1}) = \chi(\widetilde{x}_4^{-1}\widetilde{x}_3^{-1})(1-\chi(\widetilde{x}_3)) = 0&\implies \chi(\widetilde{x}_3) = 1\,,\\
        \widetilde{D}_3(r\widetilde{r}^{-1}) = \chi(\widetilde{x}_4^{-1})(\chi(\widetilde{x}_4) - 1) + 2^f = 0 & \implies \chi(\widetilde{x}_4) = -(1+2^f)^{-1}\,,\\
        \widetilde{D}_2(r\widetilde{r}^{-1}) = \chi(\widetilde{x}_2^{-1}\widetilde{x}_1^{-1})(1-\chi(\widetilde{x}_1)) = 0 &\implies \chi(\widetilde{x}_1) = 1\,,\\
        \widetilde{D}_1(r\widetilde{r}^{-1}) = \chi(\widetilde{x}_2^{-1})(\chi(\widetilde{x}_2) - 1) - 2 = 0 & \implies \chi(\widetilde{x}_2) = -1\,,\\
        \widetilde{D}_i(r\widetilde{r}^{-1}) = 0 &\implies \chi(\widetilde{x}_i) = 1\text{ for }i > 4.
    \end{align*}
    \item If \(n\) is odd and \(\mathbb{U} = \langle -1, 1+2^f\rangle\), then
    \begin{align*}
        D_1(r\widetilde{r}^{-1}) = 1 + \chi(x_1) = 0 & \implies \chi(x_1) = -1\,,\\
        D_3(r\widetilde{r}^{-1}) = \chi(x_2)^{2^f}\chi(x_2^{-1}x_3^{-1})(\chi(x_2)-1) = 0&\implies \chi(x_2) = 1\,,\\
        D_2(r\widetilde{r}^{-1}) = 2^f + \chi(x_3^{-1})(1-\chi(x_3)) =0&\implies \chi(x_3) = (1-2^f)^{-1}\,,\\
        D_i(r\widetilde{r}^{-1}) = 0 &\implies \chi(x_i) = 1\text{ for }i>3\,,\\
        \widetilde{D}_3(r\widetilde{r}^{-1}) = \chi(\widetilde{x}_3^{-1}\widetilde{x}_2^{-1})(1-\chi(\widetilde{x}_2)) =0 &\implies \chi(\widetilde{x}_2) = 1\,,\\
        \widetilde{D}_2(r\widetilde{r}^{-1}) =\chi(\widetilde{x}_3^{-1})(\chi(\widetilde{x}_3)-1) -2^f = 0 &\implies \chi(\widetilde{x}_3) = (1-2^f)^{-1}\,,\\
        \widetilde{D}_1(r\widetilde{r}^{-1}) = -\chi(\widetilde{x}_1)-1 = 0&\implies \chi(\widetilde{x}_1) = -1\,,\\
        D_i(r\widetilde{r}^{-1}) = 0 &\implies \chi(\widetilde{x}_i) = 1\text{ for }i>3\,.
    \end{align*}
\end{itemize}
In all cases we have that \(\mathrm{im}(\chi) = \mathbb{U}\), so the invariants of \(G\) must be as stated in the theorem.
\end{proof}

\begin{rmk} 
    The converse of Theorem~\ref{thm-double} is also true and follows from two results of G. Wilkes. That is, if a Demushkin group \(G\) with invariants \(n\), \(q\) and \(\mathbb{U}\) splits as a double \(F \amalg_{r = \widetilde{r}} \widetilde{F}\), then \(n\) is even and there exists a basis \(x_1,\ldots,x_{n/2}\) of \(F\) such that \(r = r_1(n/2,q,\mathbb{U})\). If such a splitting exists, then by \cite[Thm. 5.16]{wilkesRelativeCohomologyTheory2019} the free pro-\(p\) group \(F\) and the singleton containing the procyclic subgroup \(\langle r\rangle\) form a \emph{Poincaré duality pair} in dimension \(2\) (see \cite{wilkesRelativeCohomologyTheory2019} for the definition). All Poincaré duality pairs \((H,\mathcal{S})\) in dimension \(2\) where \(\mathcal{S} = \{S_0,\ldots,S_b\}\) is a set of non-trivial closed subgroups of \(H\) have been classified in \cite[Thm. 3.3]{wilkesClassificationPro$p$PD$^2$2020}. In particular, if \((F,\{\langle r\rangle\})\) is one such pair then there must exist a basis of \(F\) in which \(r\) is a defined standard Demushkin relation. In fact, a general statement about graphs of groups can be made by using \cite[Thm. 5.18]{wilkesRelativeCohomologyTheory2019}: if \(G\) is the fundamental group of a finite graph of free pro-\(p\) groups with edge groups isomorphic to \(\mathbb{Z}_p\), then \(G\) is a Demushkin group if and only if each vertex group forms a Poincaré duality pair in dimension \(2\) with respect to the incident edge groups. One also has an orientation character for Poincaré duality pairs, and the image of the orientation character of \(G\) will be generated by the images of the orientation characters of each vertex group.
\end{rmk}

\begin{cor}\label{corollary-embed-ce} Let \(G\) be a Demushkin group with invariants \(n>2\), \(q\) and \(\mathbb{U}\). If the standard Demushkin relation \(r_1(n/2,q,\mathbb{U})\) is defined, then \(G\) is a pro-\(p\) limit group.
\end{cor}
\begin{rmk} The relation \(r_1(n/2,q,\mathbb{U})\) being defined means that either \(n \in 4\bbz\) or \(q = 2\), \(n \in 4\bbz + 2\) and \([\mathbb{U}\colon \mathbb{U}^2] = 4\). If \(n \leq 2\), then either \(G \simeq \bbz/2\bbz\) or \(G \simeq \bbz_p \rtimes \bbz_p\). Since any \(2\)-generated subgroup of a pro-\(p\) limit group must be free pro-\(p\) or isomorphic to \(\bbz_p^2\), the only \(2\)-generated pro-\(p\) limit Demushkin group is \(\bbz_p^2\).
\end{rmk}
\begin{proof} Under these hypotheses, we have \(G \simeq F \amalg_{r = \widetilde{r}} \widetilde{F}\) by Theorem~\ref{thm-double}, where \(r = r_1(n/2,q,\mathbb{U})\). Since \(r\) is self-centralizing in the non-abelian free group \(F\) (here we use \(n > 2\)), in the centralizer extension \(F \amalg_{\langle r\rangle} (\langle r\rangle \times \langle a\rangle)\) the subgroup generated by \(F\) and \(aFa^{-1}\) is isomorphic to \(G\). Hence, \(G\) is a pro-\(p\) limit group.
\end{proof}

We now turn to Question~\ref{quest-free-pro-p}. While general centralizer extensions in the discrete setting are known to be residually free by a result of G. Baumslag (\cite{baumslagGeneralisedFreeProducts1962}), the best tool we have available in the pro-\(p\) world is the following result of D. Kochloukova and the second author:

\begin{thm}[{\cite[Thm. 4.1]{kochloukovaFullyResiduallyFree2010}}]\label{thm-desi-zalesskii} Let \(F = F(x_1,\ldots,x_n)\) be a free pro-\(p\) group and \(y = [x_1,x_2]t\) an element of \(F\) where  \(t\) belongs to the closed subgroup of \(F\) generated by \(x_2,\ldots,x_n\). Then, the centralizer extension \(F \amalg_{\langle y\rangle} (\langle y \rangle \times \bbz_p)\) is residually free pro-\(p\).
\end{thm}

With this result in hand, Kochloukova and the second author were able to prove that the pro-\(p\) completion of orientable surface groups of even genus are residually free, and even drop the genus restriction when \(p = 2\). From our construction on Theorem~\ref{thm-double}, one also gets an expanded class of residually free pro-\(p\) groups:

\begin{cor}\label{cor-residually-free-1} Let \(G\) be a Demushkin group with invariants \(n > 2\), \(q\) and \(\mathbb{U}\). If the standard Demushkin relation \(r_1(n/2,q,\mathbb{U})\) is defined, then \(G\) is residually free pro-\(p\).
\end{cor}
\begin{proof} Again, we can write \(G\) as a Baumslag double \(F \amalg_{r = \widetilde{r}} \widetilde{F}\) where \(r = r_1(n/2,q,\mathbb{U})\), and in Corollary~\ref{corollary-embed-ce} we've shown that \(G\) then embeds in the centralizer extension \(F \amalg_{\langle r \rangle} (\langle r \rangle \times \bbz_p)\). By exchanging \(x_1\) and \(x_2\), we can apply Theorem~\ref{thm-desi-zalesskii} to conclude that this centralizer extension (and hence \(G\) itself) is residually free pro-\(p\).
\end{proof}

The following extension is also inspired by \cite[Cor. 5.2]{kochloukovaFullyResiduallyFree2010}

\begin{cor}\label{cor-q-is-2} Let \(G\) be a pro-\(2\) Demushkin group with invariants \(n > 2\), \(q\) and \(\mathbb{U}\). If \(n\) is even, then \(G\) is a pro-\(2\) limit group and residually free pro-\(2\).
\end{cor}
\begin{proof} If \(n\) is a multiple of \(4\) or \(q = 2\) but \([\mathbb{U}\colon \mathbb{U}^2] = 4\), then we are covered by Corollary~\ref{cor-residually-free-1}. Hence we may assume \(q = 2\), \(\mathbb{U}\) is procyclic and \(n = 2d\) where \(d > 1\) is an odd integer. Therefore, there are unique integers \(s \geq 0\), \(l \geq 1\) such that \(d - 1 = 2^s(2l-1)\). Observe that the standard Demushkin relation \(r_1(4l,2,\mathbb{U})\) is defined, and let \[H = \langle y_1,\ldots,y_{4l}\mid r_1(4l,2,\mathbb{U})\rangle\,.\] It was shown in \cite{labuteClassificationDemushkinGroups1967} that the image of the orientation character of \(H\) is generated by \(\chi(y_2)\) (this also follows from the computation in the proof of Theorem~\ref{thm-double}).

Let \(U\) be an open subgroup of \(H\) of index \(2^s\) containing \(y_2\). We claim that \(U\) is isomorphic to \(G\). If \(U\) is generated by \(m\) elements, then computing Euler characteristics gives: \[\chi(U) = 2 - m = 2^s\chi(H) = 2^s(2-4l) = 2-2d\,,\] from which we conclude that \(m = 2d = n\). Since the orientation character of \(U\) is the restriction of the orientation character of \(H\) by \cite[Prop. I.18]{serreGaloisCohomology1997} and \(y_2 \in U\), we have that \(U\), \(H\) and \(G\) all share the same invariant \(\mathbb{U}\). Hence, by the classification theorem, \(U\) and \(G\) are isomorphic. Since \(H\) is a pro-\(2\) limit group by Corollary~\ref{corollary-embed-ce} and residually free pro-\(2\) by Corollary~\ref{cor-residually-free-1}, so is \(G\).
\end{proof}

We can now prove the first theorem in the introduction:

\begin{proof}[Proof of Theorem~\ref{thm-demushkin-is-limit}] Case (1) follows from Corollaries~\ref{corollary-embed-ce} and ~\ref{cor-residually-free-1}, and Case (2) follows from Corollary~\ref{cor-q-is-2}.
\end{proof}

We now move on to the study of the pro-\(2\) Demushkin group \[D = \langle x_1,x_2,x_3\mid x_1^2[x_2,x_3]\rangle \simeq \langle a,b,c\mid a^2b^2c^2 = 1\rangle\] and the proof of Theorem~\ref{thm-three-gen-is-not}. We start by treating the free quotients of \(D\):

\begin{prop}\label{prop-lyndon} Let \(F\) be a free pro-\(2\) group and \(a,b,c \in F\) be three elements satisfying \(a^2b^2c^2 = 1\). Then, \(\langle a,b,c\rangle\) is abelian.
\end{prop}
\begin{proof} Let \(H\) be the closed subgroup of \(F\) generated by \(a\), \(b\) and \(c\). It is free pro-\(p\), and any set whose image in \(H/H^2\) is a basis is a free generating set for \(H\). If \(\{a,b,c\}\) were a free generating set for \(H\), we could obtain a homomorphism \(H \to \mathrm{GL}_3(\bbz/4\bbz)\) given by \[a \mapsto A = \begin{pmatrix}
    1 & 1 & 0\\
    0 & 1 & 0\\
    0 & 0 & 1\\
\end{pmatrix}\,,\quad b \mapsto B = \begin{pmatrix}
    1 & 0 & 0\\0 & 1 & 1\\
    0 & 0 & 1
\end{pmatrix}\,,\quad c \mapsto C = \begin{pmatrix}
    1 & 0 & 1\\0 & 1  & 0\\0 & 0 & 1
\end{pmatrix}\,.\] However, one readily checks that the matrix \(A^2B^2C^2\) is not the identity matrix, yielding a contradiction. Hence, there is a linear dependence between the \(a,b,c\) modulo \(H^2\). Without loss of generality, we can then assume that \(H\) is generated by \(a\) and \(b\). If \(\{a,b\}\) were a free basis of \(H\), we could also obtain a homomorphism \(f\colon H \to \mathrm{GL}_3(\bbz/4\bbz)\) by mapping \(a\) to \(A\) and \(b\) to \(B\). Since the matrix \(A^2B^2\) has no square root in \(f(H)\), this also gives a contradiction. Therefore, \(a\) and \(b\) are linearly dependent modulo \(H^2\) and \(H\) must be procyclic.
\end{proof}

\begin{rmk} Proposition~\ref{prop-lyndon} is a pro-\(2\) version of a well known result of Lyndon on discrete free groups (\cite{lyndonEquation$a^2b^2c^2$Free1959}). The fact that we can use a Frattini argument makes the pro-\(2\) proof quite shorter than that for the corresponding abstract statement. This can also be seen as a consequence of a general result of J. Sonn (\cite{sonnEpimorphismsDemushkinGroups1974}): if \(f\colon G \to F\) is a surjective homomorphism from an \(n\)-generated Demushkin group \(G\) onto a free pro-\(p\) group of rank \(k\), then \(n \geq 2k\).
\end{rmk}

Before proving Theorem~\ref{thm-three-gen-is-not}, we need a lemma.

\begin{lem}\label{lem-action-acylindrical} Let \(\Gamma\) be the standard Bass-Serre tree associated with a centralizer extension $$(G_n = G_{n-1} \amalg_C A),$$ where \(C\) is a self-centralizing procyclic subgroup of a pro-\(p\) limit group \(G_{n-1}\) and \(A\simeq \bbz_p^k\) contains \(C\) as a direct summand. Then, for any \(H \leq G_n\), the action of \(H\) on \(\Gamma\) is \(2\)-acylindrical.
\end{lem}
\begin{proof} We recall that \(\Gamma\) is the pro-\(p\) tree with vertex set \(G_n/G_{n-1} \cup G_n/A\) and edge set \(G_n/C\), and an edge \(gC\) is connected to the vertices \(gG_{n-1}\) and \(gA\) for \(g \in G_n\). The star around a vertex \(gG_{n-1}\) consists of edges of the form \(ghC\) for \(h \in G_{n-1}\), and \(gh_1C = gh_2C\) if and only if \(h_1C = h_2C\) in \(G_{n-1}/C\). If their stabilizers \(gh_1Ch_1^{-1}g^{-1}\) and \(gh_2Ch_2^{-1}g^{-1}\) intersected non-trivially, then there would be non-zero integers \(k_1,k_2\) such that \[h_2^{-1}h_1c^{k_1}h_1^{-1}h_2 = c^{k_2}\] where \(C = \langle c \rangle\). Therefore, the \(2\)-generated subgroup \(\langle h_2^{-1}h_1, c\rangle\) of the pro-\(p\) limit group \(G_{n-1}\) must be abelian and hence \(h_2^{-1}h_1\) centralizes \(c\). This implies that \(h_1C = h_2C\) as \(C\) is self-centralized in \(G_{n-1}\), and so distinct edges in the star around \(gG_{n-1}\) have trivially intersecting stabilizers. Since any segment of length \(\geq 3\) in \(\Gamma\) contains a vertex \(gG_{n-1}\) as one if its middle vertices, it must have trivial stabilizer in \(G_n\).
\end{proof}

Moreover, since the edge stabilizers \(\mathrm{Stab}_{G_n}(e)\) for \(e \in E(\Gamma)\) are procyclic, so they are for any subgroup \(H \leq G_n\). Hence, Proposition~\ref{prop-finitely-many-stabilizers} says that if \(H \leq G_n\) is finitely generated, we have only finitely many maximal vertex stabilizers up to conjugation and hence Theorem~\ref{thm-chatzidakis} applies. We are now ready to prove the following theorem:

\begin{proof}[Proof of Theorem~\ref{thm-three-gen-is-not}] Let \(G\) be a pro-\(2\) limit group and \(\varphi\colon D \to G\) a group homomorphism. We shall proceed by induction on the weight \(n\) of \(G\), with the case \(n=0\) being a consequence of Proposition~\ref{prop-lyndon}. We now assume that \(G\) has weight \(n > 0\), and by further composing \(\varphi\) with an embedding of \(G\) into an iterated centralizer extension we can assume that \[G = G_{n-1} \amalg_C A\] where \(G_{n-1}\) is a limit pro-\(2\) group of weight \(n-1\), \(C\) is a self-centralized procyclic subgroup of \(G_{n-1}\) and \(A\) is a finitely generated free abelian pro-\(p\) group containing \(C\) as a direct factor. We denote by \(H = \varphi(D)\) the image of \(D\) in \(G\).

Let \(S(G)\) be the standard Bass-Serre pro-\(p\) tree associated with this splitting of \(G\), where we choose an orientation of \(S(G)\) such that \(d_0(gC) = gA\) and \(d_1(gC) = gG_{n-1}\). By \cite[Lemma 3.11]{horizons}, \(S(G)\) contains a unique minimal \(H\)-invariant pro-\(p\) subtree \(\Gamma\). Moreover, by a combination of the previous remark, Lemma~\ref{lem-action-acylindrical}, Proposition~\ref{prop-finitely-many-stabilizers} and Theorem~\ref{thm-chatzidakis}, $\Delta=H\backslash \Gamma$ is finite and  all vertex stabilizers are finitely generated. So by \cite[Proposition 4.4]{ZM} or \cite[Theorem 6.6]{ribesProfiniteGraphsGroups2017} $H$ is the fundamental pro-$p$ group $\pi_1(\mathcal{H},\Delta)$ of a graph of finitely generated pro-$p$ groups with cyclic edge groups. 

The stabilizers of vertices of the \(H\)-action on \(\Gamma\) are subgroups of the form \(H \cap gG_{n-1}g^{-1}\) and \(H\cap gAg^{-1}\) for some \(g \in G\). Observe we have a retraction \(\pi\colon G \to G_{n-1}\) induced from the retraction \(A \to C\) that is injective on all \(G\)-conjugates of \(G_{n-1}\). By the inductive hypothesis, the image of \(\pi\circ \varphi\) is abelian. Since \[\pi(H\cap gG_{n-1}g^{-1}) \subseteq \pi(\varphi(D))\,,\] we conclude that all vertex stabilizers of the \(H\)-action on \(\Gamma\) must be abelian.

We let \(T\) be a maximal subtree of \(\Delta\) and \(\sigma\colon \Delta \to \Gamma\) be a fundamental \(0\)-section with respect to \(T\). If \(e \in \Delta\) is an edge for which \(\mathcal{H}(e) = \mathrm{Stab}_H(\sigma(e)) = \{1\}\), then collapsing all other edges in \(\Delta\) we see that either \(H\) is freely decomposable or \(e\) is a bridge and the induced subgraph of groups on one component of \(\Delta \smallsetminus \{e\}\) has trivial fundamental group. In the first scenario, \(H\) must either be free pro-\(2\) of rank \(\leq 3\), which cannot happen by Proposition~\ref{prop-lyndon}, or be isomorphic to a free product of the form \(\mathbb{Z}_2^2 \amalg \mathbb{Z}_2\), which would lead to a surjection \[D/[D,D] \simeq \mathbb{Z}_2^2 \times \mathbb{Z}/2\mathbb{Z} \to H/[H,H] \simeq \mathbb{Z}_2^3\,,\] again a contradiction. Hence, any edge \(e \in \Delta\) for which \(\mathcal{H}(e) = 1\) is a bridge with one of its components having trivial fundamental group. However, since \(\Gamma\) is a minimal \(H\)-invariant subtree, such edges cannot exist by \cite[Lem. 3.1]{KZ2025}.

We now consider what happens if a non-trivial edge can be reduced, that is, if one of the maps \(\partial_0\colon \mathcal{H}(e) \to \mathcal{H}(d_0(e))\) or \(\partial_1\colon \mathcal{H}(e) \to \mathcal{H}(d_1(e))\) is an isomorphism. If \(\partial_1\) is an isomorphism and \(d_1(e)\) is a vertex lying on a \(G\)-orbit of a coset of \(G_{n-1}\), say \(\sigma(e) = gC\) with \(\sigma(d_1(e)) = gG_{n-1}\) for some \(g \in G\), then we claim that \(e\) is the only edge in \(\Delta\) incident to  \(d_1(e)\). Indeed, if \(e' \in \Delta\) is another edge for which \(d_1(e') = d_1(e)\), we must have \(\sigma(e') = g'C\) for some \(g' \in G\) as in Figure~\ref{fig:d0-isomorphism} and hence 
\begin{align*}
    H \cap g'C{g'}^{-1} &= \mathrm{Stab}_H(\sigma(e'))\\
    &= \mathcal{H}(e')\\
    &\leq \gamma_{e'}\mathcal{H}(d_1(e'))\gamma_{e'}^{-1} = \gamma_{e'}\mathcal{H}(d_1(e))\gamma_{e'}^{-1}\\
    &= \gamma_{e'}\gamma_e^{-1}\mathcal{H}(e)\gamma_e\gamma_{e'}^{-1}\\
    &= \gamma_{e'}\gamma_e^{-1}\mathrm{Stab}_H(\sigma(e))\gamma_e\gamma_{e'}^{-1}\\
    &= H\cap \gamma_{e'}\gamma_e^{-1}gCg^{-1}\gamma_e\gamma_{e'}^{-1}\,.
\end{align*}
Hence, if \(C = \langle c \rangle\), there are non-zero integers \(k_1\) and \(k_2\) such that \(hc^{k_1}h^{-1} = c^{k_2}\) where \[h = g'^{-1}\gamma_{e'}\gamma_e^{-1}g^{-1} \in G_{n-1}\] since \(\gamma_{e'}^{-1}g'G_{n-1} = \gamma_e^{-1}gG_{n-1}\). This implies that the \(2\)-generated subgroup \(\langle h, c\rangle\) of \(G_{n-1}\) is abelian, and since \(C\) is self-centralized in \(G_{n-1}\), this can only happen when \(h \in C\). However, we have \[h \in C \iff g'C = (\gamma_{e'}\gamma_e^{-1})gC\,,\] and thus \(\sigma(e) = gC\) and \(\sigma(e') = g'C\) are in the same \(G\)-orbit, showing that \(e = e'\).

\begin{figure}[!ht]
    \centering

\newlength{\bssep}
\setlength{\bssep}{8.2cm}
\addtolength{\bssep}{2em} 

\begin{tikzpicture}[
  >=Latex,
  font=\small,
  vfilled/.style={circle, fill=black, inner sep=1.6pt},
  vopen/.style={circle, draw, fill=white, inner sep=2.2pt},
  dashedpath/.style={densely dashed, thick},
  stable/.style={->, thick},
  midarrow/.style={
    thick,
    postaction={decorate},
    decoration={markings, mark=at position .5 with {\arrow{Latex}}}
  }
]

\def\xL{-2.85}
\def\xB{0.00}
\def\xR{1.35}
\def\xRR{3.10}

\begin{scope}[xshift=0cm]

  \node[vopen]   (L1) at (\xL,  2.6) {};
  \node[vfilled] (L0) at (\xL,  1.0) {};
  \node[vfilled] (B)  at (\xB,  0.0) {};
  \node[vfilled] (R0) at (\xR,  1.6) {};
  \node[vopen]   (R1) at (\xRR, 2.5) {};

  \node[above right=1pt] at (L1) {$d_1(\sigma(e'))$};
  \node[right=7pt]       at (L0) {$d_0(\sigma(e'))$};
  \node[below=5pt]       at (B)  {$\sigma(d_1(e)) = \sigma(d_1(e'))$};
  \node[below right=1pt] at (R0) {$d_0(\sigma(e))$};
  \node[above right=1pt] at (R1) {$d_1(\sigma(e))$};

  \draw[midarrow] (L0) -- node[midway, right=2pt] {$\sigma(e')$} (L1);
  \draw[midarrow] (R0) -- node[midway, below=2pt] {$\sigma(e)$}  (R1);

  \draw[dashedpath] (L0) -- (B);
  \draw[dashedpath] (R0) -- (B);

  \draw[stable] (B) to[bend right=35]
    node[midway, below=2pt] {$\gamma_{e'}$} (L1);

  \draw[stable] (B) to[bend left=45]
    node[midway, above=2pt] {$\gamma_{e}$} (R1);

\end{scope}

\begin{scope}[xshift=\bssep]

  \node[vopen]   (L1b) at (\xL,  2.6) {};
  \node[vfilled] (L0b) at (\xL,  1.0) {};
  \node[vfilled] (Bb)  at (\xB,  0.0) {};
  \node[vfilled] (R0b) at (\xR,  1.6) {};
  \node[vopen]   (R1b) at (\xRR, 2.5) {};

  \node[above right=1pt] at (L1b) {$g'G_{n-1}$};
  \node[right=7pt]       at (L0b) {$g'A$};

  \node[below=5pt, align=center, text width=5.0cm] at (Bb)
    {$\gamma_{e'}^{-1}g' G_{n-1}=\gamma_{e}^{-1}g G_{n-1}$};

  \node[below right=1pt] at (R0b) {$gA$};
  \node[above right=1pt] at (R1b) {$gG_{n-1}$};

  \draw[midarrow] (L0b) -- node[midway, right=2pt] {$g'C$} (L1b);
  \draw[midarrow] (R0b) -- node[midway, below=2pt] {$gC$}  (R1b);

  \draw[dashedpath] (L0b) -- (Bb);
  \draw[dashedpath] (R0b) -- (Bb);

  \draw[stable] (Bb) to[bend right=35]
    node[midway, below=2pt] {$\gamma_{e'}$} (L1b);

  \draw[stable] (Bb) to[bend left=45]
    node[midway, above=2pt] {$\gamma_{e}$} (R1b);

\end{scope}

\end{tikzpicture}

    \caption{What happens in \(\Gamma\) when a terminal vertex has two incident edges in \(\Delta\). Filled in vertices are in the image of \(\sigma\), hollowed out vertices may be not. Left picture has the labels induced from the map \(\sigma\), right picture shows the labels as \(G\)-cosets of \(A\), \(C\) and \(G_{n-1}\).}
    \label{fig:d0-isomorphism}
\end{figure}
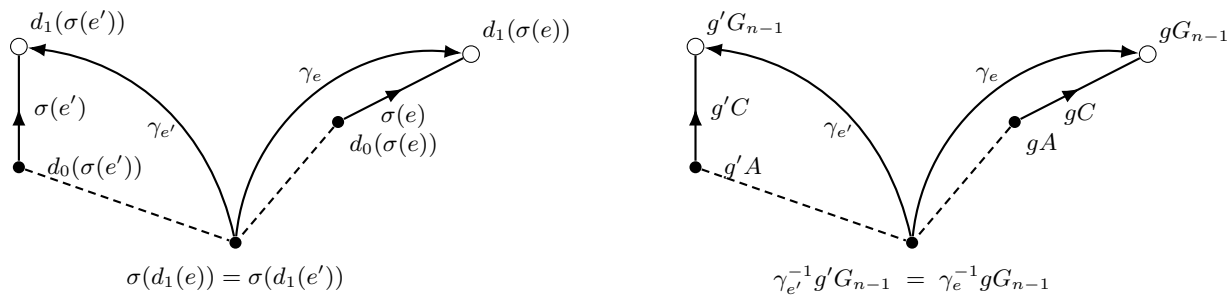

Hence, if an edge can be reduced into a vertex in the \(G\)-orbit of a coset of \(G_{n-1}\), that vertex must be hanging alone on \(\Delta\). We now claim that this is the case for all vertices in the \(G\)-orbit of a coset of \(G_{n-1}\). If \(e\) is an edge in \(\Delta\) with \(\sigma(e) = gC\), then \(\mathcal{H}(e) = H \cap gCg^{-1}\) is non-trivial. The map \(\mathcal{H}(e) \to \mathcal{H}(d_1(e))\) is given by the inclusion \(H \cap gCg^{-1} \to H\cap gG_{n-1}g^{-1}\) followed by conjugation with \(\gamma_e\). Since \(H\cap gG_{n-1}g^{-1}\) is abelian, any element in it must centralize \(gCg^{-1}\) and therefore be also contained in \(gCg^{-1}\). This shows that \(\partial_1\) is always an isomorphism, and therefore all vertices in the \(G\)-orbit of a coset of \(G_{n-1}\) must be hanging alone on \(\Delta\). However, such edges cannot exist by \cite[Lemma 3.1]{KZ2025} since $\Gamma$ is minimal, so \(\Delta\) must consist of a single vertex \(v\) in the \(G\)-orbit of a coset of \(A\), meaning that \(H = \mathcal{H}(v) = H \cap gAg^{-1}\) is abelian.
\end{proof}

We finish with a remark on fully residually free pro-\(p\) groups.

\begin{rmk} As it was shown in \cite[Thm. 7.1]{kochloukovaFullyResiduallyFree2010} a residually free pro-\(p\) group is fully residually free pro-\(p\) if and only if it is commutative transitive, that is, the centralizer of any non-trivial element is abelian. Since any closed subgroup of a Demushkin group is free pro-\(p\) or of finite index, it is an easy exercise to show that the only Demushkin group which is not commutative transitive is \(\bbz_2 \rtimes \bbz_2\) where \(\bbz_2\) acts on itself by multiplication by \(-1\). Hence, any residually free pro-\(p\) Demushkin group is also fully residually free pro-\(p\).
\end{rmk}

\bibliographystyle{alpha}
\bibliography{biblio}

\end{document}